\documentclass{amsart}
\usepackage{amsmath}
\usepackage{amssymb}
\usepackage{amsthm}
\usepackage{enumerate}
\usepackage{float, pict2e, bxeepic}

\usepackage[usenames,dvipsnames]{xcolor}

\newtheorem{theorem}{Theorem}
\newtheorem{prop}{Proposition}
\newtheorem{lemma}{Lemma}
\newtheorem{rem}{Remark}
\newtheorem{exmp}{Example}

\begin{document}
\title[Grassmann graphs of finite-rank self-adjoint operators]{Generalized Grassmann graphs associated to conjugacy classes of finite-rank self-adjoint operators}
\author{Mark Pankov, Krzysztof Petelczyc, Mariusz \.Zynel}

\address{Mark Pankov: Faculty of Mathematics and Computer Science, 
University of Warmia and Mazury, S{\l}oneczna 54, 10-710 Olsztyn, Poland}
\email{pankov@matman.uwm.edu.pl}
\address{Krzysztof Petelczyc, Mariusz \.Zynel: Faculty of Mathematics, University of Bia{\l}ystok, Cio{\l}kowskiego 1 M, 15-245 Bia{\l}ystok, Poland}
\email{kryzpet@math.uwb.edu.pl, mariusz@math.uwb.edu.pl}

\maketitle

\begin{abstract}
Two distinct projections of finite rank $m$ are adjacent if  
their difference is an operator of rank two or, equivalently, 
the intersection of their images is $(m-1)$-dimensional.
We extend this adjacency relation on other conjugacy classes of finite-rank self-adjoint operators
which leads to a natural generalization of Grassmann graphs.
Let ${\mathcal C}$  be a conjugacy class formed by finite-rank self-adjoint operators
with eigenspaces of dimension greater than $1$.
Under the assumption that operators from ${\mathcal C}$ have at least three eigenvalues
we prove that every automorphism of the corresponding generalized Grassmann graph 
is the composition of an automorphism induced by a unitary or anti-unitary operator
and the automorphism obtained from a permutation of eigenspaces with the same dimensions.
The case when the operators from ${\mathcal C}$ have two eigenvalues only 
is covered by classical  Chow's theorem which says that there are graph automorphisms induced
by semilinear automorphisms not preserving orthogonality.
\end{abstract}

\section{Introduction}
Classical Chow's theorem \cite{Chow} describes automorphisms of the Grassmann graph
whose vertices are $m$-dimensional subspaces of a vector space and 
two subspaces are adjacent (connected by an edge) if their intersection is $(m-1)$-dimensional.
It is closely related to Hua's theorem concerning adjacency preserving transformations 
of the space of matrices \cite{Wan}.
Two rectangular matrices of the same dimension are adjacent if their difference is of rank one.
Note that  the descriptions of automorphisms of so-called spine spaces \cite{PZ} puts together 
Chow's and Hua's theorems, i.e. each of them is contained in the result obtained in \cite{PZ}.

The Grassmannian of $m$-dimensional subspaces in a Hilbert space
can be identified with the conjugacy class of rank-$m$ projections.
It is natural to say that two projections (of the same rank) are adjacent if their images are adjacent subspaces
which is equivalent to the fact that the difference of these projections is an operator of rank two.
Projections can be characterized as self-adjoint idempotents in the algebra of bounded operators.
They play an important role in operator theory and mathematical foundations of quantum mechanics.
Gleason's theorem \cite{Gleason} states that pure states of quantum mechanical systems correspond to rank-one projections.
Classical Wigner's theorem \cite{Wigner} describes symmetries of the space of pure states, 
i.e. the conjugacy class of rank-one projections.
Moln\'ar \cite{M1,M2} extended Wigner's theorem on other Hilbert Grassmannians. 
Initially, such kind of  results were not related to Chow's theorem.
However, in more recent papers \cite{GeherSemrl, Geher,Pankov1}
Chow's theorem was successfully exploited to prove some Wigner-type theorems, 
see \cite{Pankov-book} for a detailed description of the topic.

In this paper, we extend Chow's theorem from the conjugacy classes of projections
(i.e. Hilbert Grassmannians) to conjugacy classes of other self-adjoint operators of finite rank.
Such operators are linear combinations of projections
and every conjugacy class formed by them is completely determined by 
the spectrum  of operators and the dimensions of their eigenspaces.

First of all, we want to understand the concept of adjacency.
It is natural to require that 
the difference of two adjacent operators $A$ and $B$ from the same conjugacy class must be of minimal rank. 
This rank is two (since the trace of $A-B$ is zero and there is no 
rank-one operator with zero trace). 
However, this condition does not guarantee that 
$A$ and $B$ have sufficiently many common eigenvectors 
(Examples \ref{exmp-pseudo-ad1} and \ref{exmp-pseudo-ad2}).
The requirement that 
the kernel and the image of $A-B$ are invariant to both $A$ and $B$ resolves this issue.
The adjacency of operators can be characterized as follows: two operators are adjacent if there is a pair of their eigenvalues such that the eigenspaces corresponding to each of these eigenvalues are adjacent as subspaces, and the remaining eigenspaces are pairwise equal.

The main result of this paper concerns automorphisms of the generalized Grassmann graph 
whose vertex set is  a conjugacy class ${\mathcal C}$ of finite-rank self-adjoint operators and 
the edge set is our adjacency relation. 
Every unitary or anti-unitary operator induces an automorphism of the graph.
If an operator from ${\mathcal C}$ has two distinct eigenspaces of the same dimension,
then the operator with these two eigenspaces transposed also belongs to the conjugacy class ${\mathcal C}$.
This way another class of automorphisms arises. 
Under the assumptions that operators from ${\mathcal C}$ have at least three eigenvalues and 
the dimension of each eigenspace is greater than one,
we prove that every automorphism of the graph is the composition of automorphisms from these two classes.
When the operators from ${\mathcal C}$ have two eigenvalues only, Chow's
theorem can be directly applied (see Example \ref{exmp-ad2}) and there are
graph automorphisms induced by semilinear automorphisms which
do not preserve orthogonality.

Connectedness and  maximal cliques in 
generalized Grassmann graphs are the main tools used in our reasonings.
The key construction is as follows.
For two fixed eigenvalues we consider subsets maximal with respect to the property that 
any two operators are connected by a path consisting of edges corresponding to these eigenvalues only. 
The family of all such connected components can be described 
as a conjugacy class of finite-rank self-adjoint operators. 
The number of eigenvalues for this class is one less than 
the number of eigenvalues of operators from ${\mathcal C}$.

There are several results (one of Hua's theorems \cite[Theorem 6.4]{Wan}  and its generalizations \cite{Herm1,Herm2})
which describe adjacency preserving maps between sets of Hermitian matrices, i.e. self-adjoint operators of finite rank.
It was noted above that two matrices are adjacent if their difference is of rank one
which immediately implies that these matrices belong to distinct conjugacy classes. 
So, these are results of different nature.

\section{Conjugacy classes of finite-rank self-adjoint operators}
Let $H$ be a complex Hilbert space and let ${\mathcal P}_{m}(H)$ be the set of all rank-$m$  projections. 
If $\lambda$ is a non-zero scalar, then we write $\lambda{\mathcal P}_{m}(H)$ for the set of all operators $\lambda P$, where $P\in {\mathcal P}_{m}(H)$.
For every closed subspace $X\subset H$ the projection on $X$ will be denoted by $P_{X}$.
Note that $P_{X^{\perp}}={\rm Id}-P_{X}$.

Two operators $A$ and $B$ on $H$  are {\it unitary conjugate} if 
there is a unitary operator $U$ on $H$ such that $B=UAU^{*}$.
Every conjugacy class ${\mathcal C}$ formed by finite-rank self-adjoint operators on $H$
is completely determined by the spectrum $\sigma=\{a_1,\dots,a_k\}$ of operators from ${\mathcal C}$ 
(each $a_i$ is real) and  the set $d=\{n_{1},\dots,n_k\}$, 
where $n_i$ is the dimension of the eigenspaces corresponding to the eigenvalue $a_i$.
For every $A\in {\mathcal C}$ the eigenspaces are mutually orthogonal
and the eigenspace corresponding to non-zero $a_i$ is finite-dimensional
(since $A$ is of finite rank).
If one of $a_i$ is zero, then the kernel of $A$ (i.e. the corresponding eigenspace) is non-trivial.
The kernel of $A$ is finite-dimensional if and only if the dimension of $H$ is finite
(since the kernel is the orthogonal complement of the image which is always finite-dimensional). 
In what follows, the conjugacy class ${\mathcal C}$ will be denoted by ${\mathcal G}(\sigma, d)$
and called the $(\sigma, d)$-Grassmannian.

For example, $\lambda{\mathcal P}_{m}(H)$ is the $(\sigma, d)$-Grassmannian, where
$\sigma=\{0,\lambda\}$ and $d=\{\dim H-m,m\}$.

If $\dim H=n$ is finite and $d=\{n\}$,
then ${\mathcal G}(\sigma, d)$ consists of a unique operator which is a scalar multiple of the identity. 
We will always exclude this case from the consideration.

If $A\in {\mathcal G}(\sigma, d)$ and $X_i$ is the eigenspace of $A$ corresponding to $a_i$,
then
$$A=\sum^{k}_{i=1}a_{i}P_{X_i}.$$
Denote by $S(d)$ the group of all permutations $\delta$ on the set $\{1,\dots,k\}$ satisfying 
$n_{\delta(i)}=n_i$.
This group is trivial if the dimensions of all $X_i$ are mutually distinct.
For every permutation $\delta \in S(d)$  the operator
$$\delta(A)=\sum^{k}_{i=1}a_i P_{X_{\delta(i)}}$$
belongs to ${\mathcal G}(\sigma, d)$.

\begin{exmp}{\rm
Suppose that $\dim H=2m$.
Then ${\mathcal G}(\sigma, d)={\mathcal P}_{m}(H)$
for $\sigma=\{0,1\}$ and $d=\{m,m\}$.
In this case, $S(d)$  coincides with $S_2$. 
If $P_{X}$ belongs to ${\mathcal P}_{m}(H)$ and $\delta$ is the non-trivial element of $S(d)$,
then $\delta(P_{X})=P_{X^{\perp}}$.
}\end{exmp}

\section{Grassmann graphs}
The conjugacy class ${\mathcal P}_{m}(H)$ can be naturally identified with 
the Grassmannian ${\mathcal G}_{m}(H)$ formed by $m$-dimensional subspaces of $H$.

Two $m$-dimensional subspaces of $H$ are called {\it adjacent} if their intersection is 
$(m-1)$-dimensional or, equivalently, their sum is $(m+1)$-dimensional.
The {\it Grassmann graph} $\Gamma_{m}(H)$ is the simple graph whose vertex set is ${\mathcal G}_{m}(H)$ 
and two $m$-dimensional subspaces are connected by an edge in this graph if they are adjacent. 
This graph is connected \cite[Proposition 2.11]{Pankov-book}.
If $m=1$ or $m=\dim H-1$, then any two distinct $m$-dimensional subspaces are adjacent, i.e.
any two distinct vertices in $\Gamma_{m}(H)$ are connected by an edge.
From this moment, we assume that $1<m<\dim H -1$.

If $\dim H$ is finite, 
then the map sending every $m$-dimensional subspace $X$  to the orthogonal complement $X^{\perp}$
is an isomorphism between the graphs $\Gamma_{m}(H)$ and $\Gamma_{\dim H-m}(H)$.
In the case when $H$ is infinite-dimensional,
two closed subspaces $X,Y\subset H$ of codimension $m$ are said to be {\it adjacent}
if $X\cap Y$ is a hyperplane  (a subspace of codimension $1$) in both $X$ and $Y$.
This is equivalent to the fact that $X^{\perp}$ and $Y^{\perp}$ are adjacent $m$-dimensional subspaces,
i.e. the Grassmann graph formed by closed subspaces of codimension $m$ is isomorphic to $\Gamma_{m}(H)$. 

The description of automorphisms of Grassmann graphs is based on the concept of semilinear automorphism.
Recall that a map $V:H\to H$ is {\it semilinear} if 
$$V(x+y)=V(x)+V(y)$$
for all $x,y\in H$ and there is an automorphism $\tau$ of the field ${\mathbb C}$ of complex numbers such that 
$$V(ax)=\tau(a)V(x)$$
for all $a\in {\mathbb C}$ and $x\in H$.
Note that there are automorphisms of ${\mathbb C}$ distinct from the identity and the complex conjugation. 
A semilinear map $V:H\to H$ is called a {\it semilinear  automorphism} of $H$ if it is bijective.

\begin{rem}\label{rem-ortho}
{\rm We will use the following fact:
every semilinear automorphism of $H$ sending orthogonal vectors to orthogonal vectors 
is a non-zero scalar multiple of a unitary or anti-unitary operator \cite[Proposition 4.2]{Pankov-book}.
}\end{rem}

Every semilinear automorphism of $H$ induces an automorphism of the Grassmann graph $\Gamma_{m}(H)$.
While this is trivial, the inverse is the famous Chow's theorem.

\begin{theorem}[Chow \cite{Chow}]
For every automorphism $f$ of the Grassmann graph $\Gamma_{m}(H)$
there is a semilinear automorphism $V:H\to H$ such that one of the following possibilities is realized:
\begin{itemize}
\item[$\bullet$] $f(X)=V(X)$ for all $X\in{\mathcal G}_{m}(H)$;
\item[$\bullet$] $\dim H=2m$ and $f(X)=V(X)^{\perp}$ for all $X\in{\mathcal G}_{m}(H)$.
\end{itemize}
\end{theorem}

\begin{rem}\label{rem-chow}{\rm
If $\Gamma_{m}(H)$ is isomorphic to $\Gamma_{m'}(H')$, then one of the following holds: 
\begin{itemize}
\item[$\bullet$] $\dim H=\dim H'$ is finite and $m'=m$ or $m'=\dim H-m$;
\item[$\bullet$] $H$ and $H'$ are infinite-dimensional and $m=m'$.
\end{itemize}
This statement is a simple consequence of argument used to prove Chow's theorem.
}\end{rem}

The proof of Chow's theorem is quite elementary and based on properties of maximal cliques of $\Gamma_{m}(H)$
called {\it stars} and {\it tops} \cite[Sections 2.3 and 2.4]{Pankov-book}.
Recall that a {\it clique} is a subset in the vertex set of a graph, where any two distinct vertices are connected by an edge. 

For any $(m-1)$-dimensional subspace $M\subset H$ 
the associated {\it star} consists of all $m$-dimensional subspaces containing $M$. 
For every $(m+1)$-dimensional subspace $N\subset H$
the associated {\it top} is ${\mathcal G}_{m}(N)$, i.e. the set of all $m$-dimensional subspaces of $N$.
The definition of stars and tops for the Grassmann graph formed by closed subspaces of codimension $m$ is similar.
The orthocomplementation map $X\to X^{\perp}$ sends stars to tops and conversely. 
While it is clear that stars and tops are cliques of $\Gamma_{m}(H)$, it can be proved
(\cite[Proposition 2.14]{Pankov-book}) that every 
maximal clique in $\Gamma_{m}(H)$ is a star or a top.

For any pair consisting of an $(m-1)$-dimensional subspace $M$ and 
an $(m+1)$-dimensional subspace $N$ such that $M\subset N$
the set of all $m$-dimensional subspaces $X$ satisfying $M\subset X\subset N$ is called a {\it line}.
So, every line is the intersection of a star and a top.
For any adjacent $m$-dimensional subspaces $X,Y$ there is a unique line containing them; 
this line consists of all $m$-dimensional subspaces $Z$ satisfying 
$$X\cap Y\subset Z \subset X+Y.$$
The following properties of maximal cliques are well-known
(see Section 2.3 and the proof of Theorem 2.15 in \cite{Pankov-book}):
\begin{enumerate}
\item[(C1)] 
The intersection of two distinct maximal cliques of $\Gamma_{m}(H)$ is empty or one element or a line.
\item[(C2)] 
For any two maximal cliques ${\mathcal X}$ and ${\mathcal Y}$ in $\Gamma_{m}(H)$  
there is a sequence of maximal cliques 
$${\mathcal X}={\mathcal X}_{0},{\mathcal X}_1,\dots,{\mathcal X}_{m}={\mathcal Y}$$
such that ${\mathcal X}_{t-1}\cap {\mathcal X}_{t}$ is a line for every
$t\in \{1,\dots,m\}$.
\end{enumerate}

\section{Adjacency relation on conjugacy classes}
Suppose that 
$$\sigma=\{a_{1},\dots,a_{k}\}\;\mbox{ and }\;  d=\{n_{1},\dots,n_k\},$$
where $a_{1},\dots,a_k$ are mutually distinct real numbers,
$n_1,\dots,n_k\in {\mathbb N}\cup \{\infty\}$ are not necessarily distinct
and only one $n_i$ corresponding to $a_i=0$ can be infinity.
According to our assumption, the case when $H$ is finite-dimensional and 
$d$ consists only of $n_1=\dim H$ is excluded.
Every operator from ${\mathcal G}(\sigma, d)$ is of rank $n$, 
where $n$ is the sum of all $n_i$ such that $a_i\ne 0$.
We take distinct $A,B\in {\mathcal G}(\sigma, d)$ and 
denote by $X_i$ and $Y_i$ the eigenspaces of $A$ and $B$ (respectively) corresponding to the eigenvalue $a_{i}$.

First of all, we observe that the rank of the operator $B-A$ is greater than $1$. 
This follows from the fact that the trace of $B-A$ is zero and the trace of every rank-one operator is non-zero.
If $B-A$ is of rank $2$, then it has non-zero real eigenvalues $-c,c$ and 
$$B=A-cP_{-}+cP_{+},$$
where $P_{-}$ and $P_{+}$ are the projections on the  $1$-dimensional eigenspaces of $B-A$ 
corresponding to $-c$ and $c$, respectively.

Next, observe that if ${\rm Ker}(B-A)$  or ${\rm Im}(B-A)$ is invariant to $A$, 
then both of these subspaces are invariant to $B$ and vice versa. 
This can be justified as follows. 
If a closed subspace $X$ is invariant to a compact self-adjoint operator $C$,
then the orthogonal complement $X^{\perp}$ is also invariant to $C$ and 
each of the subspaces $X,X^{\perp}$ has an orthogonal basis formed by eigenvectors of $C$.
Note that ${\rm Ker}(B-A)$ is the orthogonal complement of ${\rm Im}(B-A)$.
Suppose that  these subspaces are invariant to one of the operators, say $A$.
Then there is an orthogonal basis $\{e_t\}_{t\in T}$ of ${\rm Ker}(B-A)$
formed by eigenvectors of $A$. 
Since $A(x)=B(x)$ for every $x\in {\rm Ker}(B-A)$, 
each $e_t$ is an eigenvector of  both $A$ and $B$ corresponding to the same eigenvalue
which means that ${\rm Ker}(B-A)$ and, consequently, ${\rm Im}(B-A)$ are invariant to $B$.

We say that operators $A,B\in {\mathcal G}(\sigma, d)$
are {\it adjacent} if the following conditions are satisfied:
\begin{enumerate}
\item[(A1)] the rank of $B-A$ is equal to $2$;
\item[(A2)] ${\rm Im}(B-A)$ and ${\rm Ker}(B-A)$ are invariant to both $A$ and $B$.
\end{enumerate}
It was noted above that (A2) holds if one of the subspaces is invariant to at least one of the operators.

\begin{exmp}\label{exmp-pr}{\rm
Let $X$ and $Y$ be adjacent $m$-dimensional subspaces of $H$.
Denote by $X'$ and $Y'$ the $1$-dimensional orthogonal complements of 
$X\cap Y$ in $X$ and $Y$, respectively. 
Then
$$P_{X}= P_{X\cap Y}+P_{X'}\;\mbox{ and }\;P_{Y}= P_{X\cap Y}+P_{Y'}$$
which implies that 
$$P_{Y}-P_{X}=P_{Y'}-P_{X'}.$$
The image of this operator is the $2$-dimensional subspace $S=X'+Y'$.
Since
$$P_{X}(S)=X'\subset S\;\mbox{ and }\;P_{Y}(S)=Y'\subset S,$$ 
$S$ is invariant to both $P_{X}$ and $P_{Y}$ and the same holds for $S^{\perp}$.
Note that $X^{\perp}$ and $Y^{\perp}$ are adjacent and  
$$P_{Y^{\perp}}-P_{X^{\perp}}=P_{Y''}-P_{X''},$$
where $X''$ and $Y''$ are the $1$-dimensional orthogonal complements of 
$X^\perp \cap Y^{\perp}$ in $X^{\perp}$ and $Y^{\perp}$, respectively.
Since $P_{X^{\perp}}={\rm Id}-P_{X}$ and $P_{Y^{\perp}}={\rm Id}-P_{Y}$, we have
$$
P_{X'}-P_{Y'}=P_{Y''}-P_{X''}
$$
and $S=X''+Y''$.
}\end{exmp}

If $A$ and $B$ are rank-$m$ projection, then (A1) immediately implies (A2).
It fails in the general case (Examples \ref{exmp-pseudo-ad1} and \ref{exmp-pseudo-ad2}).

\begin{exmp}\label{exmp-jp}{\rm
Let us take distinct $i,j\in \{1,\dots,k\}$. 
Suppose that $X_t,Y_t$ are adjacent for $t\in \{i,j\}$ and $X_t=Y_t$ if $t\not\in \{i,j\}$.
For $t\in \{i,j\}$ we denote by $X'_t$ and $Y'_t$ the orthogonal complement of $X_{t}\cap Y_{t}$ in $X_t$ and $Y_t$, respectively. 
Since 
$$X_i+X_j=Y_i+Y_j$$
is the orthogonal complement of the sum of all $X_t=Y_t$ with $t\not\in \{i,j\}$,
we have
$$X'_i + X'_j=Y'_i + Y'_j.$$
Denote this $2$-dimensional subspace by $S$.
Then $(B-A)|_{S^{\perp}}=0$ and
the restrictions of $A$ and $B$ to $S$ are
$$a_{i}P_{X'_i}+a_{j}P_{X'_j}\;\mbox{ and }\;a_{i}P_{Y'_i}+a_{j}P_{Y'_j},$$
respectively. 
Therefore, 
$$B-A=a_{i}(P_{Y'_i}-P_{X'_i}) + a_{j}(P_{Y'_j}-P_{X'_j}).$$
The image of this operator is $S$ and (A1) holds.
The condition (A2) follows immediately from the fact that $S$ is invariant to both $A$ and $B$.
Further, we will say that  $A,B$ are $(i,j)$-{\it adjacent}.
If $a_i$ or $a_j$ is zero, then the images of $A$ and $B$ are adjacent.
Otherwise, $A$ and $B$ have the same image. 
}\end{exmp}

\begin{theorem}\label{theorem-ad}
If $A,B\in {\mathcal G}(\sigma, d)$ are adjacent,
then they are $(i,j)$-adjacent for some distinct $i,j\in \{1,\dots,k\}$.
\end{theorem}

Theorem \ref{theorem-ad} implies that there are precisely $\binom{k}{2}$ types of adjacency. 
To prove this theorem we will need the following lemma.  

\begin{lemma}\label{lemma-image}
If $T$ and $Q$ are self-adjoint operators such that
$${\rm Im}(T)\cap {\rm Im}(Q)=0$$
then 
$${\rm Im}(T+Q)={\rm Im}(T)\dotplus{\rm Im}(Q).$$
\end{lemma}

\begin{proof}
Suppose that ${\rm Im}(T)=S$ and  ${\rm Im}(Q)=U$.
The kernels of $T$ and $Q$ are $S^{\perp}$ and $U^{\perp}$, respectively.
The condition $S\cap U=0$ guarantees that 
$${\rm Ker} (T+Q)={\rm Ker}(T)\cap {\rm Ker}(Q)=S^{\perp}\cap U^{\perp}.$$
Then 
$${\rm Im}(T+Q)=({\rm Ker} (T+Q))^{\perp}=(S^{\perp}\cap U^{\perp})^{\perp}=S\dotplus U$$
which gives the claim.
\end{proof}

\begin{proof}[Proof of Theorem \ref{theorem-ad}]
It is sufficient to show that (A1) guarantees that one of the following possibilities is realized:
\begin{enumerate}
\item[$\bullet$] ${\rm Im}(A)={\rm Im}(B)$;
\item[$\bullet$] ${\rm Im}(A)\cap {\rm Im}(B)$ is $(n-1)$-dimensional.
\end{enumerate}
Then (A2) immediately  implies that the operators are $(i,j)$-adjacent
for some distinct $i,j\in \{1,\dots,k\}$.
In the first case, $a_i$ and $a_j$ both are non-zero;
in the second, precisely one of them is zero.

Suppose that 
$$\dim({\rm Im}(A)\cap {\rm Im}(B))=n-m$$
and $m\ge 2$.
The image ${\rm Im}(B)$ is the orthogonal sum of $1$-dimensional subspaces $Z_1,\dots,Z_n$
such that $B(Z_i)=Z_i$ for every $i\in \{1,\dots,n\}$.
Denote by $c_j$ the eigenvalue of $B$ corresponding to $Z_j$ (it is possible that $c_j =c_i$ for distinct $i,j$).
Then $$B=P_1 +\dots +P_{n},$$ where $P_i=c_iP_{Z_i}$. 
Show that there are $m$ distinct indices $i$ such that the sum of the corresponding $Z_i$ intersects 
${\rm Im}(A)\cap {\rm Im}(B)$ precisely in zero. 

Let us take any index set $\{i_1,\dots,i_t\}$ maximal with respect to the property that 
$Z_{i_{1}}+\dots+ Z_{i_t}$ intersects ${\rm Im}(A)\cap {\rm Im}(B)$  in zero. 
Then for every $i\not\in \{i_1,\dots,i_t\}$ we have
$$(Z_{i_{1}}+\dots+ Z_{i_t}+Z_i)\cap{\rm Im}(A)\cap {\rm Im}(B)\ne 0$$
which means that $Z_i$ is contained in 
$$(Z_{i_{1}}+\dots+ Z_{i_t})+({\rm Im}(A)\cap {\rm Im}(B))$$
and the latter subspace coincides with ${\rm Im}(B)$. 
Therefore, $t=m$.

Without loss of generality, we assume that these $m$ indices are $1,\dots,m$. 
Then, by Lemma \ref{lemma-image}, 
$$A-(P_{1}+\dots+P_{m})$$
is an operator of rank $n+m$.
The image of the operator
$$(A-B)+P_{m+1}+\dots+P_{n}$$
is contained in the subspace 
$${\rm Im}(A-B)+Z_{m+1}+\dots+Z_{n}$$
whose dimension is not greater than $n-m+2$
(recall that $A-B$ is an operator of rank $2$).
Since 
$$A-(P_{1}+\dots+P_{m})=(A-B)+P_{m+1}+\dots+P_{n},$$
we obtain that $n+m\le n-m+2$ which implies that $m\le 1$.
\end{proof}

There is one special case where the adjacency relation is trivial.
Suppose that $d=\{n_1,n_2\}$. Then $n_1 +n_2=\dim H$ ($H$ is not assumed to be finite-dimensional).
If one of $n_i$, say $n_1$, is equal to $1$, 
then the eigenspaces corresponding to $a_2$ are hyperplanes of $H$
(if $H$ is infinite-dimensional, then $a_2=0$).
Hence, it is seen that any two distinct operators from ${\mathcal G}(\sigma,d)$ are adjacent.

In the remaining cases ${\mathcal G}(\sigma,d)$ contains pairs of non-adjacent operators.

The following two examples show that (A1) does not imply (A2).

\begin{exmp}\label{exmp-pseudo-ad1}{\rm
Let $\{e_1, e_2, e_3\}$ be the canonical basis of ${\mathbb C}^3$
and let $P$ be the projection on the $1$-dimensional subspace containing $e_1$. 
Denote by $X$ and $Y$ the $1$-dimensional subspaces containing $e_1+e_2$ and $e_1+e_3$,
respectively. The matrices of the projections $P_X$ and $P_Y$ are
$$
    \frac{1}{2}\begin{bmatrix}
      1 & 1 & 0 \\
      1 & 1 & 0 \\ 
      0 & 0 & 0 
    \end{bmatrix}\qquad\text{and}\qquad
     \frac{1}{2} \begin{bmatrix}
      1& 0 & 1 \\
      0 & 0 & 0 \\
      1 & 0 & 1 
    \end{bmatrix},
$$
respectively.
The matrices of the operators 
$$A=P+P_X\;\mbox{ and }\;B=P+P_Y$$ are
$$
    \frac{1}{2}\begin{bmatrix}
      3 & 1 & 0 \\
      1 & 1 & 0 \\ 
      0 & 0 & 0 
    \end{bmatrix}\qquad\text{and}\qquad
     \frac{1}{2} \begin{bmatrix}
      3& 0 & 1 \\
      0 & 0 & 0 \\
      1 & 0 & 1 
    \end{bmatrix},
$$
respectively.
Let $U$ be the unitary operator which leaves $e_1$ fixed and transposes $e_2,e_3$.
Then $B=UAU^{*}$ and the operators $A,B$ belong to the same conjugacy class.
The images of $A$ and $B$ are the $2$-dimensional subspaces
spanned by $e_1,e_1+e_2$ and $e_1,e_1+e_3$, respectively.
The image of $$B-A=P_Y-P_{X}$$ is the $2$-dimensional subspace $X+Y$, i.e. (A1) is satisfied.
The kernel of this operator is the $1$-dimensional subspace containing $u=-e_1+e_2+e_3$
which is non-invariant to the operators $A,B$
(we have $A(u)=B(u)=-e_1$).
If $Z$ is the $1$-dimensional subspace containing $2e_1+e_2 +e_3\in X+Y$,
then $U(Z)=Z$ and the operators 
$$A'=A+P_Z\;\mbox{ and }\;B'=B+P_{Z}$$
belong to the same conjugacy class. The images of these operators coincide with ${\mathbb C}^3$.
We have $$B'-A'=B-A$$ which means that 
$B'-A'$ is an operator of rank $2$ whose kernel is the $1$-dimensional subspace containing $u=-e_1+e_2+e_3$.
This subspace is non-invariant to the operators $A',B'$
(since $P_{Z}(u)=0$ and we have $A'(u)=B'(u)=B(u)=-e_1$).
}\end{exmp}

The previous example can be generalized as follows.

\begin{exmp}\label{exmp-pseudo-ad2}{\rm
Let $C$ be a finite-rank self-adjoint operator.
We take two subspace $M,N$ such that
$$M\cap N= {\rm Im}(C)$$
is a hyperplane in both $M,N$. 
Let $X$ be a $1$-dimensional subspace of $M$ which is not contained in $M\cap N$
and let $a$ be a  non-zero real scalar.
The image of the operator
$$A=C+aP_{X}$$
is $M$.
Let $U$ be a unitary operator on $H$ which transfers $M$ to $N$
and leaves every vector from $M\cap N$ fixed.
We take $Y=U(X)$ and consider the operator
$$B=C+aP_{Y}$$
whose image is $N$.
Since $B=UAU^{*}$, the operators $A,B$ belong to the same conjugacy class. 
We have
$$B-A=a(P_{Y}-P_{X})$$
thus (A1) holds. Suppose that $X$ is non-orthogonal to $M\cap N$. 
Then $Y$ is also  non-orthogonal to $M\cap N$ and 
$$X^{\perp}\cap(M\cap N)=Y^{\perp}\cap(M\cap N)$$
is a hyperplane of $M\cap N$; we denote it  by $O$.
For every $x\in (M\cap N)\setminus O$ we have
$A(x)= C(x)+ax'$, where $x'$ is the orthogonal projection of $x$ on $X$.
Since $x$ is non-orthogonal to $X$, the vector $x'$ is non-zero
and $A(x)$ does not belong to $M\cap N$, i.e.  $x$ is not an eigenvector of $A$.
Similarly, we establish that every $x\in (M\cap N)\setminus O$ is not an eigenvector of $B$.
Therefore, an eigenvector of $C$ is an eigenvector of $A$ and $B$ only in the case when  it belongs to $O$.
This means that $A,B$ are not adjacent and the condition (A2) fails.
For example, if all eigenspaces of $C$ are $1$-dimensional, then
we can choose $X$ such that $O$ contains no non-zero eigenvector of $C$
which means that $A$ and $B$ does not contain common non-zero eigenvectors. 
The  $1$-dimensional  subspace 
$$Z=(M\cap N)^{\perp}\cap (X+Y)$$
is invariant to $U$ (since $U$ preserves $X+Y$ and $M\cap N$)
and $Z^{\perp}$ intersects $M\cap N$ precisely in $O$. 
For any non-zero real $b$ the operators
$$A'=A+bP_{Z}\;\mbox{ and }\;B'=B+bP_Z$$ 
belong to the same conjugacy class. The images of these operators coincide with $M+N$.
The image of $$B'-A'=B-A$$ is 
$X+Y$ and the kernel is $(X+Y)^{\perp}$.
Since $(X+Y)^{\perp}$ is not invariant to $A,B$ and $(X+Y)^{\perp}$ is orthogonal to $Z$,
$(X+Y)^{\perp}$ is not invariant to $A',B'$.
}\end{exmp}

\section{Graphs associated to conjugacy classes}
Denote by $\Gamma(\sigma, d)$ the simple graph whose vertices
are operators from ${\mathcal G}(\sigma,d)$ and two operators are connected by an edge if they are adjacent.
In the next section, we prove the following.

\begin{theorem}\label{theorem-conn}
The graph $\Gamma(\sigma, d)$ is  connected.
\end{theorem}

Our main result deals with automorphisms of the graph $\Gamma(\sigma,d)$.
It is natural to exclude the case when any two distinct operators from ${\mathcal G}(\sigma,d)$ are adjacent. 

\begin{exmp}\label{exmp-main}{\rm
If $U$ is a unitary or anti-unitary operator on $H$, then the transformation 
sending every $A\in {\mathcal G}(\sigma, d)$ to $UAU^*$   is an automorphism of $\Gamma(\sigma, d)$
preserving each type of adjacency.
}\end{exmp}

\begin{exmp}{\rm
Let $\delta$ be a non-trivial element of $S(d)$.
Then the transformation sending every 
$A\in{\mathcal G}(\sigma, d)$ to $\delta(A)$ is an automorphism of $\Gamma(\sigma,d)$
which does not preserve all types of adjacency.
For any unitary or anti-unitary operator $U$ on $H$
we have 
$$U\delta(A)U^{*}=\delta(UAU^{*})$$
for all $A\in {\mathcal G}(\sigma,d)$.
}\end{exmp}

\begin{theorem}\label{theorem-main}
Suppose that $k\ge 3$ and $n_i>1$ for all $i\in \{1,\dots,k\}$.
Then for every automorphism $f$ of $\Gamma(\sigma, d)$,
there are a unitary or anti-unitary operator $U$ on $H$ and a permutation  $\delta\in S(d)$
such that 
$$f(A)=U\delta(A)U^{*}$$
for all $A\in {\mathcal G}(\sigma,d)$.
In particular, every automorphism of $\Gamma(\sigma, d)$ preserving each type of adjacency
is induced by a unitary or anti-unitary operator on $H$.  
\end{theorem}

The above statement fails for $k=2$ and the
arguments used in its proof do not work when some of $n_i$ are equal
to $1$. See Example \ref{exmp-ad2} and Remark \ref{fin-rem} for more detailed explanations.

\begin{exmp}\label{exmp-ad2}{\rm
Assume that $k=2$. Then $d=\{n_1,n_2\}$ and at least one of $n_1,n_2$, say $n_1$, is finite.  
There is only one type of adjacency and $\Gamma(\sigma, d)$ is isomorphic to $\Gamma_{n_1}(H)$
(since each operator from ${\mathcal G}(\sigma,d)$ is uniquely determined by the eigenspace corresponding to $a_1$).
Let $V$ be a semilinear automorphism of $H$. 
If $A$ is an operator from ${\mathcal G}(\sigma, d)$ whose eigenspace corresponding to $a_i$ is $X_i$,
then we define $f_{V}(A)$ as the operator from ${\mathcal G}(\sigma, d)$ 
whose eigenspaces corresponding to $a_1$ and $a_2$ are 
$$V(X_1)\;\mbox{ and }\; V(X_1)^{\perp},$$
respectively.
The transformation $f_{V}$ is an automorphism of ${\mathcal G}(\sigma, d)$.
If $n_1=n_2$, then the group $S(d)$ coincides with $S_2$ and contains a unique non-identity element $\delta$;
the transformation sending every $A\in {\mathcal G}(\sigma, d)$ to $\delta(A)$  is an automorphism of 
$\Gamma(\sigma, d)$. 
It will be proved later (Lemma \ref{lemma-chow})  that
if $k=2$ and $f$ is an automorphism of $\Gamma(\sigma,d)$, then $f=f_{V}\delta$,
where $V$ is a semilinear automorphism of $H$ and $\delta\in S(d)$.
}\end{exmp}

\section{Connectedness}
Let $A,B\in {\mathcal G}(\sigma,d)$.
As above, the eigenspaces of $A$ and $B$ corresponding to $a_i$ 
are denoted by $X_i$ and $Y_{i}$, respectively. 
We say that $A,B$ are $(i,j)$-{\it connected} if there is a sequence of operators
\begin{equation}\label{eq-conn}
A=C_{0},C_{1},\dots, C_{m}=B,
\end{equation}
where $C_{t-1},C_{t}$ are $(i,j)$-adjacent for all $t\in \{1,\dots,m\}$.

\begin{lemma}\label{lemma-krzys}
Two operators from ${\mathcal G}(\sigma,d)$ are $(i,j)$-connected if and only if 
for every $p\not\in\{i,j\}$ the operators have the same eigenspace corresponding to $a_p$.
\end{lemma}

\begin{proof}
If in the sequence \eqref{eq-conn} any two consecutive  operators are $(i,j)$-adjacent,
then for every $C_t$ the eigenspace corresponding to $a_p$ with $p\not\in \{i,j\}$ is $X_p$;
in particular, $X_{p}=Y_{p}$ for all $p\not\in \{i,j\}$.

Conversely, suppose that $X_p\ne Y_p$ precisely for $p\in\{i,j\}$.
Then
$$X_{i}+X_{j}=Y_{i}+Y_{j}$$
is the orthogonal complement of the sum of all $X_{p}=Y_{p}$ with $p\not\in \{i,j\}$ 
and we denote this subspace by $M$.
If $X$ is an $n_i$-dimensional subspace of $M$ adjacent to $X_i$
and $Y$ is the orthogonal complement of $X$ in $M$,
then $Y$ is adjacent to $X_j$ and
the operator
$$a_{1}P_{X}+a_{2}P_{Y}+\sum_{p\not\in\{i,j\}}a_{p}P_{X_p}$$
is  $(i,j)$-adjacent to $A$.
The subspaces $X_{i}$ and $Y_{i}$ are connected in the Grassmann graph $\Gamma_{n_{i}}(M)$,
i.e. there is a sequence 
$$X_{i}=Z_{0},Z_{1},\dots, Z_{m}=Y_i$$
formed by $n_i$-dimensional subspaces of $M$ such that $Z_{t-1},Z_{t}$ are adjacent 
for every $t\in \{1,\dots,m\}$.
We construct a sequence of operators
$$A=C_{0},C_{1},\dots, C_{m}$$ 
such that $C_{t-1},C_{t}$ are $(i,j)$-adjacent for all $t\in \{1,\dots,m\}$
and every $Z_t$ is the eigenspace of $C_t$ corresponding to $a_i$.
For every $C_t$ the sum of the eigenspaces corresponding to $a_i$ and $a_j$ coincides with $M$. 
Since $Y_i$ is the eigenspace of $C_m$ corresponding to $a_i$, 
the eigenspace of $C_m$ corresponding to $a_j$ is $Y_j$. 
If $p\not\in\{i,j\}$, then $Y_p=X_p$ is the eigenspace of $C_m$ corresponding to $a_p$
(since $A$ and $C_m$ are $(i,j)$-connected).
Therefore, $C_m=B$ and $A,B$ are $(i,j)$-connected.
\end{proof}

\begin{proof}[Proof of Theorem \ref{theorem-conn}]
By Lemma \ref{lemma-krzys}, we need to look at the case when 
$X_i$ and $Y_i$ are distinct for at least three indices.
Suppose  that $X_i\ne Y_i$ for precisely three indices;
without loss of generality, we assume that these indices are $1,2,3$.  
Denote by $N$ the subspace
$$X_{1}+X_{2}+X_3=Y_{1}+Y_{2}+Y_3$$
which is the orthogonal complement of the sum of all $X_{i}=Y_{i}$ with $i>3$.
Let us start showing that for every $n_3$-dimensional subspace $X\subset N$ adjacent to $X_3$
there is an operator from ${\mathcal G}(\sigma,d)$ connected with $A$ and 
whose eigenspace corresponding to $a_3$ is $X$.  

To do so, observe that the $(n_3+1)$-dimensional subspace $X+X_3$ intersects $X_1+X_2$
in a certain $1$-dimensional subspace $Z'$.
We take any $n_2$-dimensional subspace $Z\subset X_1+X_2$ containing $Z'$ and adjacent to $X_2$.
There is an operator $A'\in {\mathcal G}(\sigma, d)$ which is $(1,2)$-adjacent to $A$
and such that $Z$ is the eigenspace of $A'$ corresponding to $a_2$
(see the proof of Lemma \ref{lemma-krzys}).
The $(n_2 + n_3)$-dimensional subspace $Z+X_3$ contains $X$.
If $Y$ is the orthogonal complement of $X$ in $Z+X_3$, then $Y$ is adjacent to $Z$
(since $X,X_3$ are adjacent and $Z\perp X_3$).
There is an operator $A''\in {\mathcal G}(\sigma,d)$ which is $(2,3)$-adjacent to $A'$
and such that $Y$ and $X$ are the eigenspaces of $A''$ corresponding to $a_2$ and $a_{3}$.
The operators $A$ and $A''$ are connected and we get the claim.

Now, using the connectedness of the Grassmann graph $\Gamma_{n_3}(N)$,
we construct an operator $B'\in {\mathcal G}(\sigma,d)$ connected with $A$
and such that $Y_3$ is the eigenspace of $B'$ corresponding to $a_3$. 
It must be pointed out that in the sequence of operators connecting $A$ and $B'$  
any two consecutive operators are $(1,2)$-adjacent or $(2,3)$-adjacent. 
This means that the eigenspaces of $B$ and $B'$ corresponding to $a_i$
with $i>2$ are coincident and, consequently,
$B$ and $B'$ are $(1,2)$-connected. 
So, $A$ and $B$ are connected.
 
Recursively, we establish that $A$ and $B$ are connected in the general case.
\end{proof}

Now, let us shed some light on $(i,j)$-connected components of the graph $\Gamma(\sigma,d)$, i.e.
maximal subsets of ${\mathcal G}(\sigma,d)$ with respect to the property that 
any two operators are $(i,j)$-connected. 

We take distinct $i,j\in \{1,\dots,k\}$ and
for every integer $m\le n_i$ define the new pair $(\sigma,d)^m_{-i,+j}$ as follows:
\begin{itemize}
\item if $n_i>m$, then $n_{i}$ and $n_{j}$ are replaced by $n_{i}-m$ and $n_{j}+m$, respectively;
\item if $n_i=m$, then $n_{j}$ is replaced by $n_{j}+n_i$ and 
$a_i$, $n_i$ are removed respectively from $\sigma$ and $d$.
\end{itemize}
In the case when $n_i$ is infinity, the pair $(\sigma,d)^m_{-i,+j}$ is defined for all $m\in {\mathbb N}\cup\{\infty\}$.
For any  operator $T\in {\mathcal G}((\sigma,d)^m_{-i,+j})$ we denote by ${\mathcal G}(T)$
the set of all $A\in{\mathcal G}(\sigma, d)$ such that
$$A=T+(a_i -a_j) P_{X},$$
where $X$ is an $m$-dimensional subspace in  the eigenspace of $T$ corresponding to $a_j$.

Suppose that $T\in {\mathcal G}((\sigma,d)^{n_i}_{-i,+j})$ and $M^T_j$ is the eigenspace of $T$
corresponding to $a_j$.
Then ${\mathcal G}(T)$ consists of all operators $A\in {\mathcal G}(\sigma, d)$ satisfying 
the following conditions:
\begin{itemize}
\item the sum of the eigenspaces of $A$ corresponding to $a_i$ and $a_j$ 
coincides with $M^T_j$;
\item for every $t\not\in \{i,j\}$ the eigenspaces of $A$ and $T$ corresponding to $a_t$
are coincident. 
\end{itemize}
Every operator from  ${\mathcal G}(T)$ is completely determined by the eigenspace corresponding to $a_i$.
Operators $A,B\in {\mathcal G}(T)$ are adjacent if and only if the eigenspaces of $A,B$ corresponding to $a_i$
are adjacent. 
This implies that the restriction of the graph $\Gamma(\sigma,d)$ to ${\mathcal G}(T)$
is isomorphic to the Grassmann graph $\Gamma_{n_i}(M^T_j)$ if $n_i$ is finite. 
If $n_i$ is infinity, then $n_j$ is finite and the restriction is isomorphic to $\Gamma_{n_j}(M^T_j)$.
In the case when $n_i=1$ or $n_j=1$, any two distinct operators from ${\mathcal G}(T)$ are $(i,j)$-adjacent.

The operator
$$Q=T+(a_i -a_{j})P_{M^T_j}$$
belongs to ${\mathcal G}((\sigma,d)^{n_j}_{-j,+i})$ and ${\mathcal G}(Q)={\mathcal G}(T)$.
Therefore,
$$\{{\mathcal G}(T):T\in {\mathcal G}((\sigma,d)^{n_i}_{-i,+j})\}=
\{{\mathcal G}(Q):Q\in {\mathcal G}((\sigma,d)^{n_j}_{-j,+i})\}$$
and we denote this family of subsets by $\mathfrak{G}_{ij}$.
Lemma \ref{lemma-krzys} can be reformulated as follows.

\begin{lemma}\label{lemma-ij-conn}
For any distinct $i,j\in \{1,\dots,k\}$
the family $\mathfrak{G}_{ij}$ can be characterized 
as the family of all $(i,j)$-connected components of $\Gamma(\sigma,d)$.
\end{lemma}

\begin{exmp}{\rm
Let $A\in {\mathcal G}(\sigma,d)$ and let $X_t$ be the eigenspace of $A$ corresponding to $a_t$.
For distinct $i,j\in \{1,\dots,k\}$ and distinct $p,q\in \{1,\dots,k\}$ we define the operators 
$$T\in {\mathcal G}((\sigma,d)^{n_i}_{-i,+j})\;\mbox{ and }\;Q\in {\mathcal G}((\sigma,d)^{n_p}_{-p,+q})$$
as follows:
the eigenspace of $T$ corresponding to $a_j$ is $X_i+X_j$ and 
the eigenspace of $T$ corresponding to $a_t$ with $t\not\in \{i,j\}$ is $X_{t}$;
similarly, the eigenspace of $Q$ corresponding to $a_p$ is $X_p+X_q$ and 
the eigenspace of $Q$ corresponding to $a_t$ with $t\not\in \{p,q\}$ is $X_{t}$.
Then $A$ belongs to the intersection of ${\mathcal G}(T)$ and ${\mathcal G}(Q)$.
}\end{exmp}

\begin{lemma}
Every family $\mathfrak{G}_{ij}$ is formed by mutually disjoint subsets 
and two subsets from distinct families $\mathfrak{G}_{ij}$ are disjoint or their intersection consists of one operator.
\end{lemma}

\begin{proof}
Suppose that for some operators $T,Q\in {\mathcal G}((\sigma,d)^{n_i}_{-i,+j})$
the intersection of ${\mathcal G}(T)$ and ${\mathcal G}(Q)$ contains an operator $A$.
Then the eigenspace of $T$ corresponding to $a_j$ is 
the sum of the eigenspaces of $A$ corresponding to $a_i$ and $a_j$,
the eigenspaces of $T$ and $A$  corresponding to $a_{t}$ with $t\not\in \{i,j\}$ are coincident.
The same holds for the operator $Q$ which implies that $T=Q$.

If the intersection of ${\mathcal X}\in \mathfrak{G}_{ij}$ and ${\mathcal Y}\in \mathfrak{G}_{pq}$
contains two distinct operators,
then these operators are $(i,j)$-connected and $(p,q)$-connected.
Then Lemma \ref{lemma-krzys} implies that $\{i,j\}=\{p,q\}$.
Applying arguments from the first part of the proof, we get that ${\mathcal X}={\mathcal Y}$.
\end{proof}

Denote by $\mathfrak{G}$ the union of all $\mathfrak{G}_{ij}$ and define the adjacency relation on $\mathfrak{G}$
as follows:
two subsets ${\mathcal G}(T)$ and ${\mathcal G}(Q)$ from $\mathfrak{G}$
are said to be {\it adjacent} if there is a pair of adjacent operators $A\in {\mathcal G}(T)$ and $B\in {\mathcal G}(Q)$. 

If two subsets from $\mathfrak{G}$ have a non-empty intersection, then they are adjacent. 
Indeed, if ${\mathcal G}(T)\cap{\mathcal G}(Q)=\{A\}$, 
then  each of the subsets ${\mathcal G}(T),{\mathcal G}(Q)$ contains infinitely many elements adjacent to $A$.

\begin{lemma}\label{comp-ad}
Let ${\mathcal G}(T)$ and ${\mathcal G}(Q)$ be distinct subsets from $\mathfrak{G}$.
Then the following assertions are fulfilled:
\begin{itemize}
\item[(1)] Suppose that $T$ and $Q$ belong to the same conjugacy class.
Then ${\mathcal G}(T)$ and ${\mathcal G}(Q)$ are adjacent if and only if the operators $T$ and $Q$ are adjacent.
In the case when ${\mathcal G}(T)$ and ${\mathcal G}(Q)$ are adjacent, 
there are infinitely many pairs of adjacent operators $A\in {\mathcal G}(T)$ and $B\in {\mathcal G}(Q)$. 
\item[(2)]
If $T$ and $Q$ belong to distinct conjugacy classes
and ${\mathcal G}(T),{\mathcal G}(Q)$ are adjacent,
then ${\mathcal G}(T)$ and ${\mathcal G}(Q)$ have a non-empty intersection  or 
there is a unique pair of adjacent operators $A\in {\mathcal G}(T)$ and $B\in {\mathcal G}(Q)$.
\end{itemize}
\end{lemma}

\begin{proof}
(1). Let $T,Q\in {\mathcal G}((\sigma,d)^{n_i}_{-i,+j})$. 
The first part of the statement is obvious.
Suppose that $T$ and $Q$ are adjacent and denote by $M_T$ and $M_Q$ the eigenspaces of $T$ and $Q$
corresponding to $a_j$.  Then $M_{T}$ coincides with $M_Q$ or these subspaces are adjacent. 

If $M_{T}=M_Q$, then $T,Q$ are $(p,q)$-adjacent and $\{i,j\}\cap\{p,q\}=\emptyset$.
Any pair formed by  two orthogonal subspaces of $M_T=M_Q$ whose dimensions are $n_i$ and $n_j$ 
defines two operators $A\in {\mathcal G}(T)$ and $B\in {\mathcal G}(Q)$ which are $(p,q)$-adjacent.

If $M_T$ and $M_Q$ are adjacent, then $T,Q$ are $(j,p)$-adjacent for a certain $p\not\in \{i,j\}$.
We take any $n_i$-dimensional subspace $X\subset M_T\cap M_Q$ and denote by $Y_T$ and $Y_Q$
its orthogonal complements in $M_T$ and $M_Q$, respectively. 
The pairs $X,Y_{T}$ and $X,Y_Q$ define operators $A\in {\mathcal G}(T)$ and $B\in {\mathcal G}(Q)$
(respectively). These operators are $(j,p)$-adjacent.

(2).
Suppose that $T$ and $Q$ belong to distinct conjugacy classes
and ${\mathcal G}(T),{\mathcal G}(Q)$ are disjoint and adjacent.
So, 
$$T\in {\mathcal G}((\sigma,d)^{n_i}_{-i,+j})\quad\text{ and }\quad Q\in {\mathcal G}((\sigma,d)^{n_p}_{-p,+q})$$
where $|\{i,j,p,q\}| \ge 3$. 
According to our assumption there are adjacent $A\in{\mathcal G}(T)$ and 
$B\in{\mathcal G}(Q)$. Say that they are $(u,w)$-adjacent. 
Then the following possibilities can be realized:
\begin{itemize}
\item $|\{i,j,p,q\}|=4$ and $|\{i,j,p,q,u,w\}|\in\{4,5,6\}$, 
\item $|\{i,j,p,q\}|=3$ and $|\{i,j,p,q,u,w\}|\in\{3,4,5\}$.
\end{itemize}
Consider the situation where 
$$|\{i,j,p,q\}|=4\quad\text{ and }\quad u,w\in \{i,j,p,q\}.$$
Then, in general, $A, B$ can be $(i,j)$-adjacent or $(p,q)$-adjacent, but if that is the case, 
then it is pretty easy to verify that components ${\mathcal G}(T)$, ${\mathcal G}(Q)$ have a non-empty intersection
which contradicts the assumption.
Therefore, $u\in\{i,j\}$ and $w\in \{p,q\}$.
We consider the case when $A, B$ are $(i,p)$-adjacent (the remaining cases are similar).

Suppose that there is another pair of adjacent operators $A'\in{\mathcal G}(T)$ and $B'\in{\mathcal G}(Q)$.
For every $t\in\{1,\dots,k\}$ denote by $X_t, Y_t$ the eigenspaces of $A, B$ corresponding to $a_t$ and 
by $X_t', Y_t'$ the eigenspaces of $A', B'$ corresponding to $a_t$.
Observe that 
$$Y_i=Y_i',\quad Y_j=Y_j'$$ 
as $B,B'\in{\mathcal G}(Q)$ and 
$$X_p=X_p',\quad X_q=X_q'$$ 
as $A, A'\in{\mathcal G}(T)$.

If $A',B'$ are not $(i,p)$-adjacent, then $X_t'=Y_t'$ for a certain $t\in\{i,p\}$.
Let $t=i$.
Since $X_i'\subset X_i+X_j$ as $A, A'\in{\mathcal G}(T)$, we have $Y_i =Y_i'=X_i'\subset X_i +X_j$. 
The operators $A, B$ are $(i,p)$-adjacent, so $Y_i\subset X_i +X_p$.
Consequently, 
$$Y_i\subset(X_i+X_j)\cap(X_i+X_p)=X_i$$ 
which is impossible since $X_i,Y_i$ are adjacent subspaces.

Therefore, $A',B'$ are $(i,p)$-adjacent as $A,B$ are.
Then $X_j=X_j'$. 
Since $X_i'$ is the orthogonal complement of $X_j'$ in $X_i'+X_j'=X_i+X_j$, we get that $X_i=X_i'$. 
This means that $A=A'$.
Using $Y_q=Y_q'$,  we show that $B=B'$.

The reasonings in the remaining $5$ cases are totally the same.
\end{proof}

\section{Maximal cliques}
In this section, we determine all maximal cliques of the graph $\Gamma(\sigma, d)$
and describe their properties exploited to prove Theorem \ref{theorem-main}.

\begin{lemma}\label{lemma-triangle}
For any three mutually adjacent operators from ${\mathcal G}(\sigma, d)$
there are distinct $i,j\in \{1,\dots,k\}$ such that the operators are mutually $(i,j)$-adjacent.
\end{lemma}

\begin{proof}
Let $A,B,C$ be mutually adjacent operators from ${\mathcal G}(\sigma, d)$
whose eigenspaces corresponding to $a_i$ are denoted by $X_i,Y_i, Z_i$ , respectively. 
Without loss of generality we assume that $A,B$ are $(1,2)$-adjacent
that is $X_i,Y_i$ are adjacent for $i\in \{1,2\}$ and $X_i=Y_i$ for $i>2$.
We need to show that 
$A,B,C$ are mutually $(1,2)$-adjacent.
Suppose that $C$ is $(i_A,j_A)$-adjacent to $A$ and $(i_B, j_B)$-adjacent to $B$.

In the case when $\{1,2\}\cap \{i_{A},j_{A}\}=\emptyset$, we have 
$$Z_i=X_i\ne Y_i\;\mbox{ and }\;Y_{j}=X_{j}\ne Z_{j}$$
for $i\in \{1,2\}$ and $j\in \{i_{A},j_{A}\}$. Then
$Y_i\ne Z_i$ for at least four indices $i$ which contradicts the fact that $B,C$ are adjacent.
So, $\{1,2\}\cap \{i_{A},j_{A}\}\ne \emptyset$ and the same arguments show that 
 $\{1,2\}\cap \{i_{B},j_{B}\}\ne \emptyset$.

Suppose that $i_{A},i_{B}\in \{1,2\}$ and $j_{A},j_{B}\not\in \{1,2\}$.
Then $X_{i_{A}}\ne Z_{i_A}$
and
$$Y_{j_{A}}=X_{j_{A}}\ne Z_{j_{A}}\;\;\;Y_{j_{B}}=X_{j_{B}}\ne Z_{j_{B}}.$$ 
If $j_A\ne j_B$, then $X_i\ne Z_i$ for three mutually distinct indices $i$ which contradicts the fact that 
$A,C$ are adjacent.
Therefore, $j_A=j_B$ and, without loss of generality, we can assume that $j_{A}=j_{B}=3$. 
Observe that $i_{A}$ and $i_{B}$ are distinct
(for example, if $i_{A}=i_{B}=1$, then $C$ is $(1,3)$-adjacent to both $A,B$ which means that $X_2=Z_2=Y_2$
and contradicts the fact that $A,B$ are $(1,2)$-adjacent).  

It is sufficient to consider the case when $i_{A}=2$ and $i_{B}=1$,
i.e.\ $C$ is $(2,3)$-adjacent to $A$ and $(1,3)$-adjacent to $B$, and show that this possibility cannot be realized. 
In this case, we have $Z_1 = X_1$ and $Z_2 = Y_2$. 
Therefore, $X_1$ and $Y_2$ are orthogonal.
Denote by $X_i'$ and $Y_i'$, $i\in\{1,2\}$ the $1$-dimensional orthogonal complements of 
$X_i\cap Y_i$ in $X_i$ and $Y_i$, 
respectively.  
Then $X_1\perp Y_2$ implies that 
$X'_1\perp Y'_2$ and $X_1 \perp X_2$ guarantees that
$X'_1\perp X'_2$.
Since
$$X'_1+Y'_1=X'_2+Y'_2,$$
$X'_1\perp Y'_2$ and $X'_1\perp X'_2$ show that $X'_2 = Y'_2$
which means that $X_2=Y_2$, a contradiction.
\end{proof}

\begin{lemma}\label{lemma-cliques}
If ${\mathcal X}$ is a clique of $\Gamma(\sigma, d)$, 
then there are $i,j\in \{0, 1,\dots,k\}$ such that  any two distinct operators from ${\mathcal X}$ are $(i,j)$-adjacent.
\end{lemma}

\begin{proof}
The statement is trivial if $|{\mathcal X}|=2$
and it follows immediately from Lemma \ref{lemma-triangle} if $|{\mathcal X}|=3$.
In the case when $|{\mathcal X}|\ge 4$, 
we apply  Lemma \ref{lemma-triangle} to $A,B,C$ and $A,B,D$,
where $A,B,C,D$ are distinct operators from ${\mathcal X}$.
\end{proof}

For any distinct $i,j\in\{1,\dots,k\}$ consider the pair $(\sigma,d)^1_{-i,+j}$
and an operator $T\in {\mathcal G}((\sigma,d)^1_{-i,+j})$.
Then any two distinct operators from ${\mathcal G}(T)$ are $(i,j)$-adjacent.
Indeed, if
$$A=T+(a_i -a_j) P_{X}\;\mbox{ and }\; B=T+(a_i -a_j) P_{Y},$$
where $X,Y$ are distinct $1$-dimensional subspaces in the eigenspace of $T$ corresponding to $a_j$,
then the image of the operator
$$A-B=(a_i -a_j) (P_{X}-P_{Y})$$
is a $2$-dimensional subspace contained in the sum of the eigenspaces corresponding to $a_i$ and $a_j$ 
which means that $A,B$ are $(i,j)$-adjacent.
In what follows,  we say that ${\mathcal G}(T)$ is a $(-i,+j)$-{\it clique}. 

Suppose that each of $n_i,n_j$ is greater than $1$.
For $t\in \{i,j\}$ we denote by $M^T_t$ the eigenspaces of $T$ corresponding to $a_t$ and
define $$M^{T}_{ij}=M^{T}_{i}+M^{T}_{j}$$ 
The following assertions are fulfilled:
\begin{enumerate}
\item[$\bullet$] 
for every $A\in {\mathcal G}(T)$ 
the sum of the eigenspaces corresponding to $a_i$ and $a_j$ coincides with $M^{T}_{ij}$;
\item[$\bullet$]  
the eigenspaces of all operators from ${\mathcal G}(T)$  corresponding to $a_{i}$
form the star consisting of all $n_1$-dimensional subspaces of $M^T_{ij}$ containing $M^T_{i}$;
\item[$\bullet$] the eigenspaces of all operators from ${\mathcal G}(T)$  corresponding to $a_{j}$
form the top ${\mathcal G}_{n_j}(M^{T}_{j})$.
\end{enumerate} 

\begin{rem}\label{rem-cliques}{\rm
If $n_i=1$, then the $(-i,+j)$-clique ${\mathcal G}(T)$ is an $(i,j)$-connected component 
and every $(-j,+i)$-clique intersecting ${\mathcal G}(T)$ is contained in it. 
In the case when $n_j>1$ such an $(-j,+i)$-clique is a proper subset of ${\mathcal G}(T)$.
If $n_{i}=n_{j}=1$, then every $(-i,+j)$-clique is a $(-j,+i)$-clique and conversely.
}\end{rem}

\begin{prop}\label{prop-cliques}
Every maximal clique of $\Gamma(\sigma, d)$ is a $(-i,+j)$-clique for some distinct $i,j\in\{1,\dots,k\}$.
\end{prop}

By Remark \ref{rem-cliques}, $(-j,+i)$-cliques are not maximal cliques if $n_1=1$ and $n_j>1$.

\begin{proof}[Proof of Proposition \ref{prop-cliques}]
It is sufficient to show that every clique ${\mathcal X}$ in $\Gamma(\sigma, d)$
is contained in a certain $(-i,+j)$-clique.
By Lemma \ref{lemma-cliques}, 
there are distinct $i,j\in \{1,\dots, k\}$ such that any two distinct operators from ${\mathcal X}$ are $(i,j)$-adjacent.

We take any operator $A\in {\mathcal X}$ and denote by $M$
the sum of the eigenspaces of $A$ corresponding to $a_i$ and $a_j$.
For any other operator $B\in {\mathcal X}$ the sum of the eigenspaces of $B$ corresponding to $a_i$ and $a_j$
also coincides with $M$.
Let ${\mathcal X}_i$ and ${\mathcal X}_j$ be
the sets formed by the eigenspaces of all operators from ${\mathcal X}$
corresponding to $a_i$ and $a_{j}$, respectively.
These are cliques in the Grassmann graphs $\Gamma_{n_i}(M)$ and $\Gamma_{n_j}(M)$
(respectively), i.e. each of them is a subset of a star or a top. 
Observe that $X\in {\mathcal X}_i$ if and only if $X^{\perp}\cap M\in {\mathcal X}_j$. 
Therefore, if ${\mathcal X}_i$ is contained in a star of $\Gamma_{n_i}(M)$, 
then ${\mathcal X}_j$ is contained in a top of $\Gamma_{n_j}(M)$.
In this case, ${\mathcal X}$ is a subset of a $(-i,+j)$-clique.
Similarly, if ${\mathcal X}_i$ is contained in a top of $\Gamma_{n_i}(M)$, 
then ${\mathcal X}_j$ is contained in a star of $\Gamma_{n_j}(M)$
and ${\mathcal X}$ is a subset of a $(-j,+i)$-clique.
\end{proof}

Suppose that each of $n_i,n_j$ is greater than $1$ and consider operators
$$T\in {\mathcal G}((\sigma,d)^{1}_{-i,+j}),\;\;\;Q\in {\mathcal G}((\sigma,d)^{1}_{-j,+i})$$
such that 
$$Q-T= (a_i -a_j)P_{X},$$
where $X$ is a $2$-dimensional subspace in the eigenspace of $T$ corresponding to $a_j$
(note that the dimension of this eigenspace is $n_i +1\ge 2$).
The eigenspaces of  $T$ and $Q$ corresponding to $a_t$ with $t\not\in \{i,j\}$ are coincident.
If $t\in\{i,j\}$, then we denote by $M^T_t$ and $M^Q_t$ the eigenspaces of $T$ and $Q$ (respectively)
corresponding to $a_t$.
Observe that 
$$M^Q_i=M^T_j +X\;\mbox{ and }\;M^T_j=M^Q_j +X.$$
The intersection of ${\mathcal G}(T)$ and ${\mathcal G}(Q)$  consists of all operator $A\in {\mathcal G}(\sigma,d)$
satisfying the following conditions:
 the eigenspaces of $A,T,Q$ corresponding to $a_t$ with $t\not\in \{i,j\}$ are coincident;
if $X_t$ is the eigenspace of $A$ corresponding to $a_t$ with $t\in \{i,j\}$, then 
$$M^T_i\subset X_i\subset M^Q_i\;\mbox{ and }\;M^Q_j\subset X_j\subset M^T_j.$$
In other words, the eigenspaces of all operators from ${\mathcal G}(T)\cap {\mathcal G}(Q)$
corresponding to $a_i$ and $a_j$ form lines in the Grassmannians 
${\mathcal G}_{n_i}(M)$ and ${\mathcal G}_{n_j}(M)$ (respectively),
where 
$$M=M^T_{i}+M^T_j=M^Q_{i}+M^Q_j.$$
For this reason, the intersection of ${\mathcal G}(T)$ and ${\mathcal G}(Q)$ will be called
an $(i,j)$-{\it line}.

Let ${\mathcal X}$ be an $(i,j)$-connected component of $\Gamma(\sigma,d)$.
Without loss of generality we assume that $n_i$ is finite
(since at least one of $n_i,n_j$ is finite).
It was noted in the previous section that the one-to-one correspondence between operators from ${\mathcal X}$
and their eigenspaces corresponding to $a_i$  provides an isomorphism between 
the restriction of $\Gamma(\sigma,d)$ to ${\mathcal X}$ and the Grassmann graph $\Gamma_{n_i}(M)$,
where $M$ is an $(n_i+n_j)$-dimensional subspace of $H$.
If each of $n_i,n_j$ is greater than $1$, then this isomorphism sends $(-i,+j)$-cliques to stars, $(-j,+i)$-cliques to tops
and $(i,j)$-lines to lines.

\begin{lemma}\label{int-cliques} 
If $n_i>1$ for all $i\in\{1,\dots,k\}$, then the intersection of two distinct maximal cliques of $\Gamma(\sigma,d)$ 
is empty or one operator or an  $(i,j)$-line for certain $i,j\in \{1,\dots,k\}$.
\end{lemma}

\begin{proof}
Let ${\mathcal X}$ and ${\mathcal Y}$ be distinct maximal cliques 
whose intersection contains two operators. 
These operators are $(i,j)$-adjacent for some $i,j\in \{1,\dots,k\}$. 
This implies that any two distinct operators from each of these cliques are $(i,j)$-adjacent
and there is an $(i,j)$-connected component containing ${\mathcal X}\cup{\mathcal Y}$.
The above observation together with the property (C1) from Section 3 give the claim.
\end{proof}

Similarly, the property (C2) from Section 3 gives the following.

\begin{lemma}\label{conn-lines}
Suppose that $n_i>1$ and $n_j>1$.
If ${\mathcal C}$ is a $(i,j)$-connected component of $\Gamma(\sigma,d)$
then for every maximal cliques ${\mathcal X},{\mathcal Y}\subset {\mathcal C}$ there is a sequence of 
maximal cliques 
$${\mathcal X}={\mathcal X}_{0},{\mathcal X}_1,\dots,{\mathcal X}_{m}={\mathcal Y}$$
contained in ${\mathcal C}$ and such that ${\mathcal X}_{t-1}\cap {\mathcal X}_{t}$ is a $(i,j)$-line for every
$t\in \{1,\dots,m\}$.
\end{lemma}

\section{Proof of Theorem \ref{theorem-main}}
Suppose that $n_i>1$ for all $i\in\{1,\dots,k\}$. 
Let $f$ be an automorphism of the graph $\Gamma(\sigma,d)$.
Then $f$ and $f^{-1}$ send maximal cliques of $\Gamma(\sigma,d)$ to maximal cliques
and, by Lemma \ref{int-cliques}, $f$ and $f^{-1}$ transfer lines to lines.

\begin{lemma}\label{lemma1}
For any distinct $i,j\in \{1,\dots,k\}$ there are distinct $i',j'\in \{1,\dots,k\}$
such that $f$ sends every $(i,j)$-connected component to a $(i',j')$-connected component 
and $f^{-1}$ transfers every $(i',j')$-connected component to a $(i,j)$-connected component.
\end{lemma}

\begin{proof}
Let ${\mathcal A}$ be an $(i,j)$-connected component of $\Gamma(\sigma,d)$.
We take any maximal clique ${\mathcal X}$ of $\Gamma(\sigma,d)$ contained in ${\mathcal A}$.
Any two distinct operators from this clique are $(i,j)$-adjacent.
Then $f({\mathcal X})$ is a maximal clique of $\Gamma(\sigma,d)$ and 
there are distinct $i',j'\in \{1,\dots,k\}$ such that any two distinct operators from $f({\mathcal X})$
are $(i',j')$-adjacent. 
Let ${\mathcal A}'$ be an $(i',j')$-connected component of $\Gamma(\sigma,d)$ containing $f({\mathcal X})$.
Suppose that a maximal clique ${\mathcal Y}$ intersects ${\mathcal X}$  in an $(i,j)$-line. 
Then any two operators from ${\mathcal Y}$ are $(i,j)$-adjacent and ${\mathcal Y}$  is contained in ${\mathcal A}$.
The intersection of the maximal cliques $f({\mathcal X})$ and $f({\mathcal Y})$ is an $(i',j')$-line.
This implies that $f({\mathcal Y})$ is contained in ${\mathcal A}'$. 
Using Lemma \ref{conn-lines}, we establish that $f({\mathcal A})\subset {\mathcal A}'$.
Applying  the same arguments to $f^{-1}$, we show that  $f({\mathcal A})={\mathcal A}'$.

Let ${\mathcal B}$ be an $(i,j)$-connected component of $\Gamma(\sigma,d)$.
If ${\mathcal A}$ and ${\mathcal B}$ are adjacent, then Lemma \ref{comp-ad} implies that 
$f({\mathcal B})$ is an $(i',j')$-connected component adjacent to  
$(i',j')$-connected component $f({\mathcal A})={\mathcal A}'$.
The family of all $(i,j)$-connected components can be identified with the Grassmannian
${\mathcal G}((\sigma, d)^{n_i}_{-i,+j})$ and two such components are adjacent if and only if 
the corresponding operators are adjacent (Lemma \ref{comp-ad}). 
Since the graph  $\Gamma((\sigma, d)^{n_i}_{-i,+j})$ is connected,
$f({\mathcal B})$ is an $(i',j')$-connected component.
Similarly, we show that $f^{-1}$ sends every $(i',j')$-connected component to 
a $(i,j)$-connected component.
\end{proof}

\begin{lemma}\label{lemma-dim}
If $f$ sends $(i,j)$-connected components to $(i',j')$-connected components 
then 
$$n_i =n_{i'},\;n_j=n_{j'}\;\mbox{ or }\;n_i=n_{j'},\;n_j=n_{i'}.$$
\end{lemma}

\begin{proof}
At least one of $n_i,n_j$ is finite and the same holds for $n_{i'},n_{j'}$.
Without loss of generality we assume that $n_i$ and $n_{i'}$ are finite. 
The restrictions of $\Gamma(\sigma,d)$ to $(i,j)$-connected components are isomorphic to
$\Gamma_{n_i}(M)$, where $M$ is an $(n_i + n_j)$-dimensional subspace of $H$.
Similarly, the restrictions of $\Gamma(\sigma,d)$ to $(i',j')$-connected components are isomorphic to
$\Gamma_{n_{i'}}(M')$, where $M'$ is an $(n_{i'} + n_{j'})$-dimensional subspace of $H$.
The statement follows from Remark \ref{rem-chow}.
\end{proof}

\begin{rem}\label{fin-rem}{\rm
Suppose that $n_i=1$ for a certain $i\in \{1,\dots,k\}$.
Then for every  $j\in \{1,\dots,k\}$ distinct from $i$ each $(i,j)$-connected component is an $(-i,+j)$-clique
and we are not able  to prove Lemma \ref{lemma-dim} in this case.
}\end{rem}

\begin{lemma}\label{lemma-chow}
If $k=2$, then 
there are semilinear automorphism $V$ of $H$ and a permutation $\delta\in S(d)$ such that $f=f_{V}\delta$
{\rm(}see Example \ref{exmp-ad2}{\rm)}.
\end{lemma}

\begin{proof}
Without loss of generality we assume that $n_1$ is finite.
For every $n_1$-dimensional subspace $X\subset H$
there is a unique operator $A\in{\mathcal G}(\sigma,d)$ whose eigenspace corresponding to $a_1$ is $X$
and we define $g(X)$ as the eigenspace of $f(A)$ corresponding to $a_1$. 
Then $g$ is an automorphism of the Grassmann graph $\Gamma_{n_1}(H)$
and there is a semilinear automorphism $V$ of $H$ such that for all $X\in {\mathcal G}_{n_1}(H)$
we have $g(X)=V(X)$  either $g(X)=V(X)^{\perp}$.
Then $f=f_{V}\delta$ and the permutation $\delta$ is the identity in the first case;
in the second, $n_1=n_2$ and $\delta$ is the non-trivial element of $S(d)$.
\end{proof}

\begin{flushleft}
\bf
The case $k=3$
\end{flushleft}

\noindent
First of all, we consider the case when $f$ preserves each type of adjacency 
and show that it is induced by a unitary or anti-unitary operator.
There are distinct $i,j\in \{1,2,3\}$ such that $n_i +n_j$ is finite and 
$n_i +n_j\ne n_{t}$, where $t\in \{1,2,3\}$ is distinct from $i,j$.
 Without loss of generality we assume that $n_1+n_{2}$ is finite and $n_1+n_{2}\ne n_{3}$.
Since $f$ preserves the family of $(1,2)$-connected components, 
it induces an automorphism $g$ of the graph $\Gamma((\sigma,d)^{n_1}_{-1,+2})$.
The assumption $n_1+n_{2}\ne n_{3}$ guarantees that $g=f_V$ for a certain semilinear automorphism $V$ of $H$
(see Lemma \ref{lemma-chow}).

For every $T\in {\mathcal G}((\sigma,d)^{n_1}_{-1,+2})$ we denote by $M_T$ the eigenspace of $T$ corresponding to $a_2$. Then $M^{\perp}_T$ is the eigenspace of $T$ corresponding to $a_3$.
The eigenspaces of $g(T)$ corresponding to $a_2$ and $a_3$ are $V(M_T)$ and $V(M_T)^\perp$, respectively.
 For every subspace $X\subset M_T$ of dimension not less than $n_1$ we denote by $[X]$
the set of all operators $A\in {\mathcal G}(T)$ whose eigenspaces corresponding to $a_1$ are contained in $X$.
If $X$ is $n_1$-dimensional, then $[X]$ consists of one operator.

Suppose that $X$ is a hyperplane of $M_T$. 
We choose an operator $Q\in {\mathcal G}((\sigma,d)^{n_1}_{-1,+2})$ adjacent to $T$
and such that $M_T\cap M_Q=X$.
For every $A\in [X]$ there is $B\in {\mathcal G}(Q)$ which is $(2,3)$-adjacent to $A$.
The operators
$$f(A)\in {\mathcal G}(T)\;\mbox{ and }\; f(B)\in {\mathcal G}(Q)$$
are $(2,3)$-adjacent.
Therefore, $f(A)$ and $f(B)$ have the same eigenspace corresponding to $a_1$
which is contained in 
$$V(M_T)\cap V(M_Q)=V(M_T\cap M_Q)=V(X),$$
i.e. $f(A)$ belongs to $[V(X)]$. 
So, 
\begin{equation}\label{eq-proof1}
f([X])\subset [V(X)] 
\end{equation}
if $X$ is a hyperplane of $M_T$. 

Now, we take any operator $A\in {\mathcal G}(T)$ and assume that $X$ is the eigenspace of $A$
corresponding to $a_1$. Then $[X]=\{A\}$. 
There are hyperplanes $X_1,\dots, X_{n_2}$ of $M_T$ whose intersection coincides with $X$.
Since 
$$f([X_t])\subset [V(X_t)]$$ 
for every  $t\in \{1,\dots,n_2\}$, we obtain \eqref{eq-proof1} again.
This means that $V(X)$ is the eigenspace of $f(A)$ corresponding to $a_1$.
Similarly, we show that $V(Y)$ is the eigenspace of $f(A)$ corresponding to $a_2$
if $Y$ is  the eigenspace of $A$ corresponding to $a_2$.
The subspaces $V(X)$ and $V(Y)$ are orthogonal. 
Since $A$ and $T$ are arbitrary taken,
the latter implies that $V$ sends orthogonal vectors to orthogonal vectors.
Then $V$ is a scalar multiple of a unitary or anti-unitary operator (Remark \ref{rem-ortho})
and we can assume that $V$ is unitary or anti-unitary. 
The eigenspaces of $A$ and $f(A)$ corresponding to $a_3$ are $(X+Y)^{\perp}$
and $(V(X)+V(Y))^{\perp}$, respectively.
Since $V$ is unitary or anti-unitary, it sends  the first subspace to the second,
i.e. $f(A)=UAU^*$.

Consider the general case.
For every $t\in \{1,2,3\}$ we take distinct $i,j\in \{1,2,3\}$ satisfying $\{i,j,t\}=\{1,2,3\}$.
By Lemma \ref{lemma1}, there are distinct $i',j'\in \{1,2,3\}$ such that 
$f$ sends $(i,j)$-connected components to $(i',j')$-connected components.
Denote by $\delta(t)$ the element of $\{1,2,3\}$ satisfying $\{i',j',\delta(t)\}=\{1,2,3\}$.
Then for any distinct $p,q\in  \{1,2,3\}$ the automorphism $f$ transfers 
$(p,q)$-connected components to $(\delta(p),\delta(q))$-connected components. 
We need to show that $n_{i}=n_{\delta(i)}$ for every $i\in \{1,2,3\}$, 
i.e. $\delta$ belongs to $S(d)$.

If $n_1=n_2=n_3$, then every permutation on the set $\{1,2,3\}$ belongs to $S(d)$.
In the case when $n_1,n_2,n_3$ are mutually distinct, Lemma \ref{lemma-dim} shows that 
$\delta$ preserves every $2$-element subset of  $\{1,2,3\}$ which means that it is the identity.
If, for example, $n_1=n_2\ne n_3$, then, by Lemma \ref{lemma-dim}, $\delta$ preserves $\{1,2\}$
which implies that $\delta(3)=3$ and $\delta$ is the identity or the transposition $(1,2)$.

So, $\delta\in S(d)$ and $\delta^{-1} f$ is an automorphism of $\Gamma(\sigma,d)$ preserving each type of adjacency.
Then $\delta^{-1} f$ is induced by a unitary or anti-unitary operator which gives the claim.

\begin{flushleft}
\bf
The case $k\ge 4$
\end{flushleft}

\noindent
In this case Theorem \ref{theorem-main} will be proved recursively.
Without loss of generality we assume that 
\begin{equation}\label{eq-n}
n_1\ge n_2\ge\dots\ge n_k.
\end{equation}
Suppose that $f$ sends $(1,2)$-connected components to $(i,j)$-connected components.
By Lemma \ref{lemma-dim},  we have 
$$n_1=n_i,\;n_2=n_j\;\mbox{ or }\;n_1=n_j,\;n_2=n_i.$$
In the first case, we set $\delta_1 =(1,i)(2,j)$;
in the second, $\delta_1=(1,j)(2,i)$. 
For each of these cases we have $\delta_1 \in S(d)$
and $g=\delta_{1} f$ is an automorphism of $\Gamma(\sigma, d)$ preserving $(1,2)$-adjacency. 
Therefore, $g$ preserves the family of $(1,2)$-connected components and 
induces a bijective transformation $h$ of ${\mathcal G}((\sigma,d)^{n_1}_{-1,+2})$. 
It follows from Lemma \ref{comp-ad} that $h$ is an automorphism of the graph $\Gamma((\sigma,d)^{n_1}_{-1,+2})$.

Note that $(\sigma,d)^{n_1}_{-1,+2}=(\sigma', d')$, where
$$\sigma'=\{a_2,\dots,a_k\}\;\mbox{ and }\; d'=\{n_1+n_2, n_3,\dots,n_k\}.$$ 
Suppose that there are a unitary or anti-unitary operator $U$ and a permutation $\delta_2\in S(d')$
(we consider $\delta_2$ as a permutation on the set \{2,\dots,k\})
such that 
$$h(T)=U\delta_2(T)U^{*}$$
for every $T\in {\mathcal G}((\sigma,d)^{n_1}_{-1,+2})$. 
The condition \eqref{eq-n} guarantees that 
$$n_1+n_2>n_t$$
for every $t\ge 3$. Therefore, $\delta_2(2)=2$.

Also, for any distinct $p,q\ge 3$ the automorphism $g$ sends 
$(p,q)$-adjacent operators to $(\delta_2(p),\delta_2(q))$-adjacent operators.
This means that $g$ transfers $(2,3)$-adjacent operators to $(t,s)$-adjacent operators, 
where $t\in \{1,2\}$ and $s\ge 3$.

Suppose that $t=2$, i.e. $g$ sends $(2,3)$-adjacent operators to $(2, s)$-adjacent operators with $s\ge 3$.
As in the proof for $k=3$, we establish the following: 
if $X$ is the eigenspace of $A\in {\mathcal G}(\sigma,d)$ corresponding to $a_1$, 
then $U(X)$ is the eigenspace of $g(A)$ corresponding to $a_1$. 
This immediately implies that $U(Y)$ is the eigenspace of $g(A)$ corresponding to $a_2$ 
if $Y$ is the eigenspace of $A\in {\mathcal G}(\sigma,d)$ corresponding to $a_2$. 

In the case when $t=1$, i.e. $g$ sends $(2,3)$-adjacent operators to $(1, s)$-adjacent operators with $s\ge 3$,
the same arguments show that if $X$ is the eigenspace of $A\in {\mathcal G}(\sigma,d)$ corresponding to 
$a_p$, $p\in \{1,2\}$ then $U(X)$ is the eigenspace of $g(A)$ corresponding to $a_{3-p}$.
In particular, we obtain that $n_1=n_2$.

Then $g(A)=U\delta(T)U^{*}$
for every $A\in {\mathcal G}(\sigma,d)$. 
The permutation $\delta$ coincides with $\delta_2$ for every element greater than $2$.
In the first case, $\delta$ leaves $1$ and $2$ fixed; in the second, it transposes them.

\end{document}